\documentclass[12pt]{amsart2020}
\usepackage{amscd,amssymb}

\newtheorem{theorem}{Theorem}[section]

\newtheorem{proposition}[theorem]{Proposition}
\newtheorem{corollary}[theorem]{Corollary}

\theoremstyle{definition}

\newtheorem{remark}[theorem]{Remark}

\begin{document}

\title[Noncommutative invariants of dihedral groups]
{Noncommutative invariants\\ of dihedral groups}

\author[Vesselin Drensky, Boyan Kostadinov]
{Vesselin Drensky, Boyan Kostadinov}
\address{Institute of Mathematics and Informatics,
Bulgarian Academy of Sciences, 1113 Sofia, Bulgaria}
\email{drensky@math.bas.bg}
\email{boyan.sv.kostadinov@gmail.com}

\thanks{The research of the second named author was supported
by the National Programme of Ministry of Education and Science
``Young Scientists and Post-Doctoral Researchers - 2'' and
partially supported by
Grant KP-06-N-32/1 of the Bulgarian National Science Fund.}


\maketitle

\begin{abstract}
Let $A_2({\mathfrak M})$ and $L_2({\mathfrak A}^2)$ be, respectively, the 2-generated free metabelian associative and Lie algebra over the field of complex numbers.
In the associative case we find a finite set of generators of the algebra $A_2({\mathfrak M})^{D_{2n}}$ of the dihedral group of order $2n$, $n\geq 3$.
In the Lie case we find a minimal system of generators as a ${\mathbb C}[x,y]^{D_{2n}}$-module
of the $D_{2n}$-invariants $L_2'({\mathfrak A}^2)^{D_{2n}}$ in the commutator ideal $L_2'({\mathfrak A}^2)$ of $L_2({\mathfrak A}^2)$.
In both cases we compute the Hilbert (or Poincar\'e) series of the algebras $A_2({\mathfrak M})^{D_{2n}}$ and $L_2({\mathfrak A}^2)^{D_{2n}}$.
\end{abstract}

\section{Introduction}
In the sequel we assume that $K$ is a field of characteristic 0. We shall consider the case $K=\mathbb C$ only although our main results hold for any sufficiently large field $K$.
In classical commutative invariant theory the general linear group $GL_d({\mathbb C})$ acts on the $d$-dimensional complex vector space $V_d$ with basis $\{v_1,\ldots,v_d\}$
and this action induces an action on the algebra of polynomial functions ${\mathbb C}[X_d]={\mathbb C}[x_1,\ldots,x_d]$.
For a subgroup $G$ of $GL_d({\mathbb C})$ the algebra of $G$-invariants ${\mathbb C}[X_d]^G$ consists of all polynomials which are fixed under the action of $G$.
In one of the main branches of noncommutative invariant theory one replaces the algebra ${\mathbb C}[X_d]$ with an algebra which shares some of the important properties of ${\mathbb C}[X_d]$.
Among the candidates for such an algebra are the free associative algebra ${\mathbb C}\langle X_d\rangle={\mathbb C}\langle x_1,\ldots,x_d\rangle$, the free Lie algebra $L_d$
or the $d$-generated relatively free algebra of a variety of (associative, Lie, Jordan or other nonassociative) algebras, $d\geq 2$.
In this case it is more convenient to assume that $GL_d({\mathbb C})$ acts directly on the vector space ${\mathbb C}X_d$ with basis $X_d$ instead of on $V_d$.
Going back to classical invariant theory this means that $GL_d({\mathbb C})$ acts on the symmetric algebra $S(X_d)$ of ${\mathbb C}X_d$ instead on the algebra of polynomial functions ${\mathbb C}[X_d]$.
We refer to the surveys by Formanek \cite{F} and the first named author of the present paper \cite{D2} for more details.

\subsection{Varieties of associative algebras.}
Let $\mathfrak V$ be a variety of unitary associative algebras over a field $K$ of characteristic 0 and let $A_d({\mathfrak V})$ be the relatively free algebra of $\mathfrak V$ freely generated by $X_d$, $d\geq 2$.
The general linear group $GL_d(K)$ acts canonically on the vector space $KX_d$ and this action is extended diagonally on the whole algebra $A_d({\mathfrak V})$:
\[
g(f(x_1,\ldots,x_d))=f(g(x_1),\ldots,g(x_d)),\,f\in A_d({\mathfrak V}),g\in GL_d(K).
\]
For a subgroup $G$ of $GL_d(K)$ the algebra of $G$-invariants is
\[
A_d({\mathfrak V})^G=\{f(X_d)\in A_d({\mathfrak V})\mid g(f(X_d))=f(X_d)\text{ for all }g\in G\}.
\]
There are many equivalent conditions on the variety $\mathfrak V$ to the fact that the algebra $A_d({\mathfrak V})^G$ is finitely generated for all finite subgroups $G$ of $GL_d(K)$.
See \cite[Section 4.7]{KS} or \cite[Theorem 2.1]{DD} for the complete statement. In particular, due to the results by Kharchenko \cite{Kh}, Lvov \cite{Lv} and Anan'in \cite{A}
the following theorem holds.

\begin{theorem}\label{finite generation for all finite G}
Let $\mathfrak V$ be a variety of unitary associative algebras. The following
conditions on $\mathfrak V$ are equivalent. If some of them is
satisfied for some $d_0 \geq 2$, then all of them hold for all $d\geq 2$:

{\rm (i)} The algebra $A_d({\mathfrak V})^G$ is finitely generated for every
finite subgroup $G$ of $GL_d(K)$.

{\rm (ii)} The variety $\mathfrak V$ satisfies the polynomial identity
\[
[x_1,x_2,\ldots,x_2]x_3^n[x_4,x_5,\ldots,x_5] = 0
\]
for sufficiently long commutators and $n$ large enough.
\end{theorem}

Here and further $[z_1,z_2]=z_1z_2-z_2z_1=z_1\,\text{ad}(z_2)$ is the commutator of $z_1$ and $z_2$ and the longer commutators are left normed:
\[
[z_1,\ldots,z_{n-1},z_n]=[[z_1,\ldots,z_{n-1}],z_n],\,n\geq 3.
\]
In particular Theorem \ref{finite generation for all finite G} holds for the metabelian variety $\mathfrak M$ defined by the polynomial identity
\[
[x_1,x_2][x_3,x_4]=0.
\]

The metabelian variety plays a key role in the theory of PI-algebras. By a theorem of Kemer \cite[Corollary 5]{Ke}
if an associative algebra satisfies the Engel identity
\[
[x_2,x_1,\ldots,x_1]=x_2\,\text{ad}^n(x_1)=0
\]
then it is Lie nilpotent, i.e. $[x_1,\ldots,x_m]=0$ for some $m$.
As a consequence there is a dichotomy for varieties of associative algebras.
The variety $\mathfrak V$ is either Lie nilpotent or contains $\mathfrak M$ as a subvariety.

In the sequel we shall denote the free generators of $A_2({\mathfrak M})$ by $x$ and $y$.
The first main result of our paper gives a system of generators of the algebra $A_2({\mathfrak M})^{D_{2n}}$ of the invariants of the dihedral group $D_{2n}$ of order $2n$
under the natural action of $D_{2n}$ as the group of symmetries of the regular $n$-gon, $n\geq 3$. We also compute the Hilbert series of the algebra $A_2({\mathfrak M})^{D_{2n}}$.
Our proof uses essentially the well known description of the algebra of invariants ${\mathbb C}[x,y]^{D_{2n}}$.

\subsection{Varieties of Lie algebras.}
The metabelian variety of Lie algebras ${\mathfrak A}^2$ is defined by the polynomial identity
\[
[[x_1,x_2],[x_3,x_4]]=0.
\]
Zelmanov \cite{Z} proved that a Lie algebra satisfying the Engel identity is nilpotent.
This implies that as in the case of associative algebras there is a dichotomy.
A variety is either nilpotent or contains ${\mathfrak A}^2$ as a subvariety.
Concerning invariant theory of relatively free algebras, by a result of the first named author of this paper \cite{D1}
if $G$ is a nontrivial finite subgroup of $GL_d(K)$, $d\geq 2$, then the algebra of $G$-invariants $L_d({\mathfrak V})^G$ in the relatively free algebra $L_d({\mathfrak V})$
of the variety of Lie algebras $\mathfrak V$ is not finitely generated if and only if ${\mathfrak A}^2\subseteq \mathfrak V$. This result was generalized by Bryant and Papistas \cite{BP}
for any finite subgroup of the automorphism group of $L_d({\mathfrak V})$. A positive result in this direction is obtained by Drensky, F{\i}nd{\i}k and \"O\u{g}\"u\c{s}l\"u \cite[Corollary 5]{DFO}.
They showed that if $G<GL_d(K)$ is a finite group then $L_d'({\mathfrak A}^2)^G$ is a finitely generated $K[X_d]^G$-module, where the variables $x_i\in X_d$ act on $L_d'({\mathfrak A}^2)$
by the adjoint action $x_i\to\text{ad}(x_i)$.

Since the free metabelian Lie algebra $L_d({\mathfrak A}^2)$ is naturally embedded into the free metabelian associative algebra $A_d({\mathfrak M})$,
when $K=\mathbb C$ the results for the algebra of invariants $A_2({\mathfrak M})^{D_{2n}}$
allow easily to obtain a minimal set of generators of $L_2'({\mathfrak A}^2)^{D_{2n}}$ as a ${\mathbb C}[x,y]^{D_{2n}}$-module.
As in the associative case, we compute the Hilbert series of the algebra $L_2({\mathfrak A}^2)^{D_{2n}}$.

\section{Preliminaries}
From now on we shall work on the field of complex numbers.

\subsection{Invariant theory of dihedral groups}
The facts on invariant theory of dihedral groups which we shall need in the sequel are well known. We include the proofs for completeness of the exposition.
Recall the Endlichkeitssatz of Emmy Noether and the theorem of Chevalley-Shephard-Todd.

\begin{theorem}\label{Endlichkeitssatz}\cite{N}
For any finite subgroups $G$ of $GL_d({\mathbb C})$ the algebra of invariants ${\mathbb C}[X_d]^G$ is finitely generated
and has a homogeneous system of generators of degree $\leq\vert G\vert$.
\end{theorem}

Recall that an element $g\in GL_d({\mathbb C})$ of finite order is a (pseudo)reflection if it has an eigenvalue 1 of multiplicity $d-1$
and another eigenvalue of multiplicity 1 which is a root of unity. The group $G<GL_d({\mathbb C})$ is called a reflection group if it is generated by reflections.
For our purposes we need a stronger than in the original papers \cite{Ch, ST} form of the theorem of Chevalley-Shephard-Todd.

\begin{theorem}\label{Chevalley-Shephard-Todd} \cite[Theorem 4.2.5]{Sp} The following properties of the finite subgroup $G$ of $GL_d({\mathbb C})$ are equivalent:
\begin{enumerate}
\item $G$ is a finite reflection group;
\item ${\mathbb C}[X_d]$ is a free graded module over ${\mathbb C}[X_d]^G$ with a finite basis;
\item ${\mathbb C}[X_d]^G$ is generated by $d$ algebraically independent homogeneous elements.
\end{enumerate}
\end{theorem}

\begin{corollary} \cite[Corollary 4.2.12]{Sp}
Let $G$ be a reflection group and let ${\mathbb C}[X_d]^G = \mathbb{C}[f_1,\ldots,f_n]$, where $f_i$ is homogeneous of degree $d_i$.
Then up to order the integers $d_i$ are uniquely determined by $G$.
The order $\vert G\vert$ of $G$ is equal to $\prod\limits_{i = 1}^n d_i$ and the number of reflections in $G$ is equal to $\sum\limits_{i = 1}^n (d_i - 1)$.
\end{corollary}

\begin{proposition}\label{number of free generators}\cite[Proposition 3.6]{Hu}
If $G$ is a reflection group then the free ${\mathbb C}[X_d]^G$-module ${\mathbb C}[X_d]$ is freely generated by $\vert G\vert$ homogeneous elements.
\end{proposition}

The dihedral group $D_{2n}$, $n\geq 3$, acts on the two dimensional real vector space $xOy$ as the group of symmetries of the regular $n$-gon.
It is generated by a rotation by angle $\displaystyle \frac{2\pi}{n}$ around the origin and a reflection with respect to the axis $Ox$.
According to Klein (\cite[Kapitel II]{K1} in the German original or \cite[Chapter II]{K2} in the English translation) already Riemann used the following change of the coordinate system:
\[
u=x+iy,v=x-iy.
\]
Then $D_{2n}$ is generated by the rotation
\[
\rho: u \longrightarrow \xi u,\, v \longrightarrow \bar{\xi} u,\text{ where }\xi = e^{\frac{2 \pi i}{n}},
\]
and the reflection
\[
\tau: u \longleftrightarrow v.
\]
Till the end of the paper we shall work with the generators of $D_{2n}$ with respect to the coordinates $(u,v)$ instead of with the coordinates $(x,y)$.

\begin{proposition}\label{dihedral invariants}
The algebra of invariants ${\mathbb C}[u,v]^{D_{2n}}$ is generated by $uv$ and $u^n + v^n$.
\end{proposition}

\begin{proof}
We shall give an elementary proof which uses only the Endlich\-keits\-satz of Emmy Noether.
Obviously the invariants of $\tau$ in ${\mathbb C}[u,v]$ are linear combinations of
\[
u^{p+q}v^p+u^qv^{p+q}=(uv)^p(u^q + v^q).
\]
Since
\[
\varrho((uv)^p(u^q + v^q))=(uv)^p(\xi^pu^q + \bar{\xi}^qv^q)
\]
we derive that the invariants of $\varrho$ occur for $q \equiv 0\,(\text{mod }n)$.
Hence the vector space ${\mathbb C}[u,v]^{D_{2n}}$ has a basis
\[
\{(uv)^p(u^{qn} + v^{qn})\mid p,q\geq 0\}
\]
and the algebra ${\mathbb C}[u,v]^{D_{2n}}$ is generated by
\[
\{uv,u^{qn}+v^{qn}\mid q\geq 1\}.
\]
Since $\vert D_{2n}\vert=2n$, the Endlichkeitssatz of Emmy Noether gives that ${\mathbb C}[u,v]^{D_{2n}}$ is generated by
\[
uv,u^n+v^n,u^{2n}+v^{2n}.
\]
This completes the proof because
\[
u^{2n}+v^{2n}=(u^n+v^n)^2-2(uv)^n
\]
and $u^{2n}+v^{2n}$ belongs to the subalgebra of ${\mathbb C}[u,v]$ generated by $uv$ and $u^n+v^n$.
\end{proof}

Recall that the Hilbert (or Poincar\'e) series of a graded vector space $W=\sum\limits_{i\geq 0}W^{(i)}$ with finite dimensional homogeneous components $W^{(i)}$ is the formal power series
\[
H(W,t)=\sum_{i\geq 0}\dim(W^{(i)})t^i.
\]

\begin{proposition}\label{free generators of C[u,v] as a module}
The polynomial algebra $\mathbb{C}[u, v]$ is a free $\mathbb{C}[u, v]^{D_{2n}}$-module with free generators
\[
W=(1, u, u^2, \ldots, u^n, v, v^2,\ldots, v^{n-1}).
\]
\end{proposition}

\begin{proof}
The fact that $\mathbb{C}[u, v]$ is a free $\mathbb{C}[u, v]^{D_{2n}}$-module follows immediately from Theorem \ref{Chevalley-Shephard-Todd}.
By Proposition \ref{dihedral invariants} $\mathbb{C}[u, v]^{D_{2n}}$ is a polynomial algebra generated by two elements of degree 2 and $n$, respectively.
Hence the Hilbert series of $\mathbb{C}[u, v]$ and $\mathbb{C}[u, v]^{D_{2n}}$ are
\[
H(\mathbb{C}[u, v],t)=\frac{1}{(1-t)^2}\text{ and }H(\mathbb{C}[u, v]^{D_{2n}},t)=\frac{1}{(1-t^2)(1-t^n)}.
\]
By Proposition \ref{number of free generators} the $\mathbb{C}[u, v]^{D_{2n}}$-module $\mathbb{C}[u, v]$ is generated by $2n$ homogeneous polynomials
$z_1,\ldots,z_{2n}$ and
\[
\mathbb{C}[u, v]=z_1\mathbb{C}[u, v]^{D_{2n}}\oplus\cdots\oplus z_{2n}\mathbb{C}[u, v]^{D_{2n}}.
\]
Hence
\[
H(\mathbb{C}[u, v],t)=\sum_{j=1}^{2n}t^{\deg(z_j)}H(\mathbb{C}[u, v]^{D_{2n}},t),
\]
\[
\sum_{j=1}^{2n}t^{\deg(z_j)}=\frac{H(\mathbb{C}[u, v],t)}{H(\mathbb{C}[u, v]^{D_{2n}},t)}=(1+t)(1+t+\cdots+t^{n-1})
\]
and $\mathbb{C}[u, v]$ is generated as a $\mathbb{C}[u, v]^{D_{2n}}$-module by one element of degree 0 and of degree $n$ and 2 elements of each degree $1,2,\ldots,n-1$.
Let
\[
Z=(z_0,z_1,z_1',\ldots,z_{n-1},z_{n-1}',z_n)
\]
be such a system of free generators,
\[
\deg(z_0)=0,\,\deg(z_i)=\deg(z_i')=i,\, i=0,1,\ldots,n-1,\, \deg(z_n)=n.
\]

The next step of our proof is to prove that all elements of degree $\leq n$ in $\mathbb{C}[u,v]$ can be expressed as linear combinations of the system $W$ with coefficients from $\mathbb{C}[u, v]^{D_{2n}}$.
The monomials of degree $\leq n$ in ${\mathbb C}[u,v]$ are of the form $u^av^b$, $a+b\leq n$.
\begin{itemize}
\item If $a=b$, then $u^av^a=1\cdot(uv)^a$.
\item If $a>b$, then $u^av^b=u^{a-b}\cdot(uv)^b$, $a-b\leq n$.
\item If $a<b$, then $u^av^b=v^{b-a}\cdot(uv)^b$ if $b-a<n$. Since $a+b\leq n$, if $b-a=n$, then $(a,b)=(0,n)$ and $u^av^b=v^n=1\cdot (u^n+v^n)-u^n\cdot 1$.
\end{itemize}

Hence all elements of degree $\leq n$ belong to the ${\mathbb C}[u,v]^{D_{2n}}$-module generated by $W$.
In particular, this holds to the system $Z$ of $2n$ free generators of the ${\mathbb C}[u,v]^{D_{2n}}$-module ${\mathbb C}[u,v]$.
Hence there exists a $2n\times 2n$ matrix $T$ with entries in $\mathbb{C}[u, v]^{D_{2n}}$ such that $Z^t=TW^t$, where  $Z^t$ and $W^t$ are, respectively, the transposed of $Z$ and $W$.
On the other hand, since $Z$ is a system of free generators of the ${\mathbb C}[u,v]^{D_{2n}}$-module ${\mathbb C}[u,v]$, there exists a $2n\times 2n$ matrix $S$ such that $W^t=SZ^t$. This implies that
$Z^t=T(SZ^t)$ and that the matrix $TS$ is the identity matrix. Hence the matrix $T$ is invertible, i.e. $W$ is a free generating set of the free ${\mathbb C}[u,v]^{D_{2n}}$-module ${\mathbb C}[u,v]$.
\end{proof}

The next corollary follows immediately from Proposition \ref{free generators of C[u,v] as a module} because
\[
{\mathbb C}[u_1,v_1,u_2,v_2]\cong {\mathbb C}[u_1,v_1]\otimes_{\mathbb C}{\mathbb C}[u_2,v_2].
\]

\begin{corollary}\label{action on C[u1,v1,u2,v2]}
Let $D_{2n}$ act on ${\mathbb C}[u_1,v_1]$ and ${\mathbb C}[u_2,v_2]$ in the same way as on ${\mathbb C}[u,v]$. Then ${\mathbb C}[u_1,v_1,u_2,v_2]$
is a free ${\mathbb C}[u_1,v_1]^{D_{2n}}\otimes_{\mathbb C} {\mathbb C}[u_2,v_2]^{D_{2n}}$-module freely generated by
\[
u_1^au_2^c, u_1^av_2^d,v_1^bu_2^c,v_1^bv_2^d,\,0\leq a,c\leq n,1\leq b,d\leq n-1.
\]
\end{corollary}

\subsection{Free metabelian associative algebra}\label{associative preliminary}
The description of the free metabelian associative algebra $A_d({\mathfrak M})$ is well known, see e.g.
\cite[Proposition 3]{La} or \cite[Theorem 5.2.1]{D3}.

\begin{proposition}\label{properties of free associative metabelian algebra}
The free metabelian algebra $A_d({\mathfrak M})$ has a basis consisting of all
\[
x_1^{a_1}\cdots x_d^{a_d},\, x_1^{a_1} \cdots x_d^{a_d} [x_{i_1},\dots, x_{i_n}], \, a_1,\ldots,a_d\geq 0,i_1 > i_2 \leq \dots \leq i_n.
\]
It satisfies the identities
\[
x_{\varphi(1)}\cdots x_{\varphi(m)}[x_{m+1},x_{m+2},x_{\psi(m+3)},\ldots,x_{\psi(m+n)}]
\]
\[
=x_1\cdots x_m[x_{m+1},x_{m+2},x_{m+3},\ldots,x_{m+n}],\,m\geq 0,n\geq 3,
\]
where $\varphi$ and $\psi$ are permutations of $\{1,\ldots,m\}$ and $\{m+3,\ldots,m+n\}$, respectively.
\end{proposition}

\begin{corollary}\label{associative case rank 2}
{\rm (i)} In the special case of $d=2$ the algebra $A_2({\mathfrak M})$ has a basis with respect to the set $\{u,v\}$
\[
u^av^b,\, u^av^b[v,u]u^cv^d,\, a,b,c,d\geq 0.
\]

{\rm (ii)} The commutator ideal $A_2'({\mathfrak M})$ is a free ${\mathbb C}[u_1,v_1,u_2,v_2]$-module generated by $[v,u]$, where
${\mathbb C}[u_1,v_1]$ and ${\mathbb C}[u_2,v_2]$ act on $A_2'({\mathfrak M})$ by multiplication by $u$ and $v$ from the left and the right, respectively.
\end{corollary}

\begin{proof}
As in Proposition \ref{properties of free associative metabelian algebra}, the metabelian identity implies that
\[
z_2z_1z_3^{\varepsilon}[z_4,z_5]=z_1z_2z_3^{\varepsilon}[z_4,z_5],\,[z_1,z_2]z_3^{\varepsilon}z_5z_4=[z_1,z_2]z_3^{\varepsilon}z_4z_5,\,\varepsilon=0,1.
\]
Hence we can rewrite the basis of the algebra $A_2({\mathfrak M})$ with respect to the set $\{u,v\}$ in the desired form and this proves (i).
The same equations imply that $A_2'({\mathfrak M})$ is a ${\mathbb C}[u_1,v_1,u_2,v_2]$-module. The basis of $A_2'({\mathfrak M})$ given in (i)
gives the freeness of the module.
\end{proof}

\subsection{Free metabelian Lie algebra}\label{Lie preliminary}
As in the associative case, the structure of free metabelian Lie algebra $L_d({\mathfrak A}^2)$ is also well known, see e.g. \cite[Section 4.7.1]{B}.
It has a basis
\[
x_1,\ldots,x_d,\,[x_{i_1},\dots, x_{i_n}], \, i_1 > i_2 \leq \dots \leq i_n.
\]
Clearly, $L_d({\mathfrak A}^2)$ may be considered as the Lie subalgebra of $A_d({\mathfrak M})$ generated by $x_1,\ldots,x_d$ with respect to the commutator operation.
In the special case $d=2$ the basis of $L_2({\mathfrak A}^2)$ consists of
\[
u,v,\,[v,u]\text{ad}^a(u)\text{ad}^b(v),\,a,b\geq 0,
\]
and $L_d'({\mathfrak A}^2)$ is a free ${\mathbb C}[u,v]$-module generated by $[v,u]$ with respect to the action
\[
f(u,v):w\longrightarrow wf(\text{ad}(u),\text{ad}(v)),\quad w\in L_d'({\mathfrak A}^2), f(u,v)\in{\mathbb C}[u,v].
\]

\section{Main results}
\subsection{Associative case}\label{associative main}

Applying Corollary \ref{associative case rank 2} (ii) we define an isomorphism of ${\mathbb C}[u_1,v_1,u_2,v_2]$-modules
\[
\nu:A_2'({\mathfrak M})\longrightarrow{\mathbb C}[u_1,v_1,u_2,v_2]\text{ by }\nu(u^av^b[v,u]u^cv^d)=u_1^av_1^bu_2^cv_2^d.
\]
Clearly, $\nu$ is also an isomorphism of ${\mathbb C}[u_1,v_1]^{D_{2n}}\otimes_{\mathbb C}{\mathbb C}[u_2,v_2]^{D_{2n}}$-modules.
This allows to restate Corollary \ref{action on C[u1,v1,u2,v2]} in the following way.

\begin{proposition}\label{basis of A2' as module of dihedral invariants}
Let $D_{2n}$ act on ${\mathbb C}[u_1,v_1]$ and ${\mathbb C}[u_2,v_2]$ in the same way as on ${\mathbb C}[u,v]$. Then $A_2'({\mathfrak M})$
is a free ${\mathbb C}[u_1,v_1]^{D_{2n}}\otimes_{\mathbb C} {\mathbb C}[u_2,v_2]^{D_{2n}}$-module freely generated by
\[
u^a[v,u]u^c, u^a[v,u]v^d,v^b[v,u]u^c,v^b[v,u]v^d,\,0\leq a,c\leq n,1\leq b,d\leq n-1.
\]
\end{proposition}

The next proposition is the key step in the description of the algebra of invariants $A_2({\mathfrak M})^{D_{2n}}$.

\begin{proposition}\label{invariants in commutator ideal-associative case}
{\rm (i)} A polynomial $f(u,v)\in A_2'({\mathfrak M})$ belongs to $A_2'({\mathfrak M})^{D_{2n}}$ if and only if
it is fixed under the action of the rotation $\varrho$ and the reflection $\tau$ which generate the dihedral group $D_{2n}$, i.e.
\[
f(\xi u,\bar{\xi} v)=f(u,v),f(v,u)=f(u,v),\text{ where }\xi=e^{\frac{2 \pi i}{n}}.
\]
The polynomial $h(u_1,v_1,u_2,v_2)=\nu(f(u,v))\in{\mathbb C}[u_1,v_1,u_2,v_2]$ belongs to ${\mathbb C}[u_1,v_1,u_2,v_2]^{D_{2n}}$ if and only if
\[
\varrho(h(u_1,v_1,u_2,v_2))=h(\xi u_1,\bar{\xi}v_1,\xi u_2,\bar{\xi}v_2)=h(u_1,v_1,u_2,v_2),
\]
\[
\tau(h(u_1,v_1,u_2,v_2))=h(v_1,u_1,v_2,u_2)=-h(u_1,v_1,u_2,v_2).
\]

{\rm (ii)} The ${\mathbb C}[u_1,v_1]^{D_{2n}}\otimes_{\mathbb C}{\mathbb C}[u_2,v_2]^{D_{2n}}$-module $A_2'({\mathfrak M})^{D_{2n}}$ is free
with a set of free generators
\[
u^a[v,u]u^{n-a}-v^a[v,u]v^{n-a},\,a=0,1,\ldots,n,\,u^n[v,u]u^n-v^n[v,u]v^n,
\]
\[
u^a[v,u]v^a-v^a[v,u]u^a,\,a=1,\ldots,n-1.
\]
\end{proposition}

\begin{proof}
(i) A polynomial $f(u,v)\in A_2'({\mathfrak M})$ belongs to $A_2'({\mathfrak M})^{D_{2n}}$ if and only if it is fixed under the action of the generators of $D_{2n}$:
\[
\varrho(f(u,v))=f(\xi u,\bar{\xi}v)=f(u,v)\text{ and }\tau(f(u,v))=f(v,u)=f(u,v).
\]
Since
\[
\varrho([v,u])=[\xi v,\bar{\xi}u]=[v,u]\text{ and }\tau([v,u])=[u,v]=-[v,u],
\]
we obtain that
\[
h(u_1,v_1,u_2,v_2)=\nu(f(u,v))=\nu(\varrho(f(u,v)))
\]
\[
=\varrho(h(u_1,v_1,u_2,v_2))=h(\xi u_1,\bar{\xi}v_1,\xi u_2,\bar{\xi}v_2),
\]
\[
h(u_1,v_1,u_2,v_2)=\nu(f(u,v))=\nu(\tau(f(u,v)))
\]
\[
=-\nu(f(v,u))=-h(v_1,u_1,v_2,u_2).
\]

(ii) We shall describe the image of $A_2'({\mathfrak M})^{D_{2n}}$ in ${\mathbb C}[u_1,v_1,u_2,v_2]$ under the action of $\nu$.
By Corollary \ref{action on C[u1,v1,u2,v2]} the ${\mathbb C}[u_1,v_1]^{D_{2n}}\otimes_{\mathbb C}{\mathbb C}[u_2,v_2]^{D_{2n}}$-module ${\mathbb C}[u_1,v_1,u_2,v_2]$ is freely generated by
\[
u_1^au_2^c, u_1^av_2^d,v_1^bu_2^c,v_1^bv_2^d,\,0\leq a,c\leq n,1\leq b,d\leq n-1.
\]
Since
\[
\varrho(u_1^av_1^bu_2^cv_2^d)=\xi^{a+c-b-d}u_1^av_1^bu_2^cv_2^d,\,\tau(u_1^av_1^bu_2^cv_2^d)=u_1^bv_1^au_2^dv_2^c
\]
and the polynomials in ${\mathbb C}[u_1,v_1]^{D_{2n}}$ and in ${\mathbb C}[u_2,v_2]^{D_{2n}}$
are fixed under the action of $\varrho$ and $\tau$, from the action of $\varrho$ we obtain that
$\nu(A_2'({\mathfrak M})^{D_{2n}})$ is spanned by monomials $u_1^av_1^bu_2^cv_2^d$ subject to the condition
\[
a+c\equiv b+d\text{ (mod }n).
\]
The action of $\tau$ on ${\mathbb C}[u_1,v_1,u_2,v_2]$ gives that the ${\mathbb C}[u_1,v_1]^{D_{2n}}\otimes_{\mathbb C}{\mathbb C}[u_2,v_2]^{D_{2n}}$-module
$\nu(A_2'({\mathfrak M})^{D_{2n}})$ is generated by
\[
u_1^au_2^{n-a}-v_1^av_2^{n-a},\,a=0,1,\ldots,n,\,u_1^nu_2^n-v_1^nv_2^n,
\]
\[
u_1^av_2^a-v_1^au_2^a,\,a=1,\ldots,n-1.
\]
(Here $v_i^n$, $i=1,2$, are not in the free generating set of the module ${\mathbb C}[u_1,v_1,u_2,v_2]$ but they can be expressed as
$v_i^n=1\cdot(u_i^n+v_i^n)-u_i^n$.) Since
\[
A_2'({\mathfrak M})^{D_{2n}}=\nu^{-1}({\mathbb C}[u_1,v_1,u_2,v_2]^{D_{2n}}),
\]
\[
\nu^{-1}(u_1^au_2^c-v_1^av_2^b)=v^a[v,u]u^c-v^a[v,u]v^c,
\]
\[
\nu^{-1}(u_1^av_2^a-v_1^au_2^a)=u^a[v,u]v^a-v^a[v,u]u^a
\]
this completes the proof of (ii).
\end{proof}

Now we state the main result on the invariants of the dihedral group $D_{2n}$ acting on the free metabelian algebra $A_2({\mathfrak M})$.

\begin{theorem}\label{invariants in the associative case}
{\rm (i)} For $n\geq 3$ the algebra $A_2({\mathfrak M})^{D_{2n}}$ is generated by $uv+vu$, $u^n+v^n$ and
\[
u^a[v,u]u^{n-a}-v^a[v,u]v^{n-a},\,a=0,1,\ldots,n,\,u^n[v,u]u^n-v^n[v,u]v^n,
\]
\[
u^a[v,u]v^a-v^a[v,u]u^a,\,a=1,\ldots,n-1.
\]

{\rm (ii)} The Hilbert series of $A_2({\mathfrak M})^{D_{2n}}$ is
\[
H(A_2({\mathfrak M})^{D_{2n}},t)=\frac{1}{(1-t^2)(1-t^n)}
\]
\[
+\frac{1}{(1-t^2)^2(1-t^n)^2}\left((n+1)t^{n+2}+\frac{t^4(1-t^{2n})}{1-t^2}\right).
\]
\end{theorem}

\begin{proof}
(i) We can split the set of generators of the algebra $A_2({\mathfrak M})^{D_{2n}}$ in two parts -- the generators of the invariants in the polynomial algebra
\[
{\mathbb C}[u,v]^{D_{2n}}\cong A_2({\mathfrak M})^{D_{2n}}/A_2'({\mathfrak M})^{D_{2n}}
\]
and the generators of $A_2'({\mathfrak M})^{D_{2n}}$ as a ${\mathbb C}[u,v]^{D_{2n}}$-bimodule. Lifting the generators $uv$ and $u^n+v^n$ of ${\mathbb C}[u,v]^{D_{2n}}$
to $A_2({\mathfrak M})^{D_{2n}}$ we obtain the polynomials $uv+vu$, $u^n+v^n$. The set of the generators of $A_2'({\mathfrak M})^{D_{2n}}$ as a ${\mathbb C}[u,v]^{D_{2n}}$-bimodule
is given in Proposition \ref{invariants in commutator ideal-associative case} (ii).

(ii) We have
\[
H(A_2({\mathfrak M})^{D_{2n}},t)=H(A_2({\mathfrak M})^{D_{2n}}/A'_2({\mathfrak M})^{D_{2n}},t)+H(A'_2({\mathfrak M})^{D_{2n}},t),
\]
\[
H(A_2({\mathfrak M})^{D_{2n}}/A'_2({\mathfrak M})^{D_{2n}},t)=H({\mathbb C}[u,v]^{D_{2n}},t)=\frac{1}{(1-t^2)(1-t^n)}.
\]
Since $A'_2({\mathfrak M})^{D_{2n}}$ is a free ${\mathbb C}[u,v]^{D_{2n}}$-module freely generated by a finite set of elements $z_i$,
\[
H(A'_2({\mathfrak M})^{D_{2n}},t)=H^2({\mathbb C}[u,v]^{D_{2n}},t)\sum_it^{\text{deg}z_i}
\]
\[
=\frac{1}{(1-t^2)^2(1-t^n)^2}\sum_it^{\text{deg}z_i}.
\]
By Proposition \ref{invariants in commutator ideal-associative case} (ii) the set of the free generators of the ${\mathbb C}[u,v]^{D_{2n}}$-bimodule
$A'_2({\mathfrak M})^{D_{2n}}$ consists of:
\begin{itemize}
\item $n+1$ generators $u^a[v,u]u^{n-a}-v^a[v,u]v^{n-a}$, $a=0,1,\ldots,n$, of degree $n+2$;
\item One generator $u^n[v,u]u^n-v^n[v,u]v^n$ of degree $2n+2$;
\item One generator $u^a[v,u]v^a-v^a[v,u]u^a$ of degree $2a+2$ for each $a=1,\ldots,n-1$.
\end{itemize}
Hence
\[
\sum_it^{\text{deg}z_i}=(n+1)t^{n+2}+\sum_{a=1}^nt^{2a+2}=(n+1)t^{n+2}+\frac{t^4(1-t^{2n})}{1-t^2}
\]
and this completes the proof.
\end{proof}

\begin{remark}
The generating set of the algebra $A_2({\mathfrak M})^{D_{2n}}$ is not minimal because
the commutator $[uv+vu,u^n+v^n]$ is of degree $n+2$ and can be expressed as a linear combination of the $n+1$ generators
\[
u^a[v,u]u^{n-a}-v^a[v,u]v^{n-a},\,a=0,1,\ldots,n.
\]
It is easy to see that if we remove one of these generators, we shall obtain a minimal generating set of the algebra $A_2({\mathfrak M})^{D_{2n}}$.
\end{remark}

\begin{remark}\label{general field K}
The generators of $A_2({\mathfrak M})^{D_{2n}}$ in Theorem \ref{invariants in the associative case} are given as polynomials in $u$ and $v$.
We can express them as polynomials in the generators $x$ and $y$ of $A_2({\mathfrak M})$. For the action of the dihedral group on the two-dimensional vector space
with basis $\{x,y\}$ over a field $K$ of characteristic 0 we need that $\displaystyle \cos\left(\frac{2\pi}{n}\right),\displaystyle \sin\left(\frac{2\pi}{n}\right)\in K$.
Then the generators of $A_2({\mathfrak M})^{D_{2n}}$ in Theorem \ref{invariants in the associative case} expressed in terms of $x$ and $y$
generate also the $D_{2n}$-invariants of the two-generated free metabelian associative algebra over $K$.
\end{remark}

\subsection{Lie case}
Our considerations will be similar to those in Subsection \ref{dihedral invariants} but are much simpler because by
Subsection \ref{Lie preliminary} the commutator ideal $L_2'({\mathfrak A}^2)$ is a free ${\mathbb C}[u,v]$-module generated by $[v,u]$.
Proposition \ref{free generators of C[u,v] as a module} immediately gives the following.

\begin{proposition}\label{basis of L2' as module of dihedral invariants}
The commutator ideal $L_2'({\mathfrak A}^2)$ is a free ${\mathbb C}[u,v]^{D_{2n}}$-module freely generated by
\[
[v,u]\text{\rm ad}^a(u),\,a=0,1,\ldots,n,\,[v,u]\text{\rm ad}^b(v),\, b=1,\ldots,n-1.
\]
\end{proposition}

The following theorem is the Lie analogue of Theorem \ref{invariants in the associative case}.

\begin{theorem}\label{invariants in the Lie case}
{\rm (i)} The ${\mathbb C}[u,v]^{D_{2n}}$-module $L_2'({\mathfrak A}^2)^{D_{2n}}$ is freely generated by
\[
[v,u](\text{\rm ad}^n(u)-\text{\rm ad}^n(v)).
\]

{\rm (ii)} The Hilbert series of $L_2({\mathfrak A}^2)^{D_{2n}}$ is
\[
H(L_2({\mathfrak A}^2)^{D_{2n}},t)=\frac{t^{n+2}}{(1-t^2)(1-t^n)}.
\]
\end{theorem}

\begin{proof} (i) As in Subsection \ref{associative main}, we define an isomorphism of ${\mathbb C}[u,v]$-modules
\[
\nu:LA_2'({\mathfrak A}^2)\longrightarrow{\mathbb C}[u,v]\text{ by }\nu([v,u]\text{ad}^a(u)\text{ad}^b(v))=u^av^b.
\]
Again $\nu$ is an isomorphism also of ${\mathbb C}[u,v]^{D_{2n}}$-modules.
As in Proposition \ref{invariants in commutator ideal-associative case} $f(u,v)\in L_2'({\mathfrak A}^2)^{D_{2n}}$ if and only if
it is fixed under the action of the rotation $\varrho$ and the reflection $\tau$. The polynomial
$h(u,v)=\nu(f(u,v))\in{\mathbb C}[u,v]$ belongs to ${\mathbb C}[u,v]^{D_{2n}}$ if and only if
\[
\varrho(h(u,v))=h(\xi u,\bar{\xi}v)=h(u,v),\,\tau(h(u,v))=h(v,u)=-h(u,v).
\]
Taking into account the action of $\varrho$, the image $\nu(L_2'({\mathfrak A}^2)^{D_{2n}})$ of $L_2'({\mathfrak A}^2)^{D_{2n}}$
is spanned as a ${\mathbb C}[u,v]^{D_{2n}}$-module by monomials
\[
u^a,v^b,\,a,b\equiv 0\text{ (mod }n),\,a=0,1,\ldots,n,\,b=1,\ldots,n-1.
\]
Now the action of $\tau$ gives that the only possibility is
\[
u^n-v^n (=2u^n-1\cdot (u^n+v^n)).
\]
Hence  $L_2'({\mathfrak A}^2)^{D_{2n}}$ is generated as a ${\mathbb C}[u,v]^{D_{2n}}$-module by
\[
\nu^{-1}(u^n-v^n)=[v,u](\text{ad}^n(u)-\text{ad}^n(v)).
\]

(ii) Since all $D_{2n}$-invariants of $L_2({\mathfrak A}^2)$ are in the commutator ideal, and $L_2'({\mathfrak A}^2)^{D_{2n}}$
is generated as a free ${\mathbb C}[u,v]^{D_{2n}}$-module by one polynomial of degree $n+2$, we obtain that
\[
H(L_2({\mathfrak A}^2)^{D_{2n}},t)=t^{n+2}H({\mathbb C}[u,v]^{D_{2n}},t)=\frac{t^{n+2}}{(1-t^2)(1-t^n)}.
\]
\end{proof}

\begin{remark}
As in Remark \ref{general field K}, Theorem \ref{invariants in the Lie case}
can be restated for the $D_{2n}$-invariants of the two-generated free metabelian Lie algebra over a sufficiently large field $K$ of characteristic 0.
\end{remark}

\end{document}